\documentclass[12pt]{amsart}

\newcommand{\Q}{\mathbb{Q}}
\newcommand{\R}{\mathbb{R}}
\newcommand{\C}{\mathbb{C}}

\newcommand{\GL}{\mathrm{GL}}
\newcommand{\SL}{\mathrm{SL}}
\newcommand{\Mn}{\mathrm{M}_n}

\newcommand{\tr}{\operatorname{tr}}
\newcommand{\Csr}{\ensuremath{\mathrm{C}_r^*}}

\newcommand{\U}{\mathcal{U}}

\newcommand{\scal}[2]{\langle #1,#2\rangle}

\newtheorem{theorem}{Theorem}[]

\newtheorem{prop}[theorem]{Proposition}

\theoremstyle{definition}
\newtheorem{definition}[theorem]{Definition}

\theoremstyle{remark}
\newtheorem{remark}[theorem]{Remark}

\begin{document}

\title{C*-algebraic normalization and Godement-Jacquet factors}
\dedicatory{To Gestur \'Olafsson on occasion of his 65\textsuperscript{th} birthday, in admiration.}

\author{Pierre Clare}
\address{Department of Mathematics, College of William \& Mary, P.O. Box 8795, Williamsburg, VA 23187-8795}
\curraddr{}
\email{prclare@wm.edu}
\thanks{}


\date{\today}

\begin{abstract}
We observe how certain distribution on $\GL(n)$ that generates Godement-Jacquet $\gamma$-factors appears naturally in the C*-algebraic normalization process for standard intertwining integrals on $\SL(n+1)$.
\end{abstract}

\maketitle

\section{Introduction}

The purpose of this note is to point out a relationship between well-known factors introduced by R. Godement and H. Jacquet in the 1970's and distributions that occur in the study of the tempered dual of reductive groups in the framework of Noncommutative Geometry. The functions $\gamma$ considered here appear in the functional equations satisfied by the zeta functions associated with irreducible representations of $\GL(n)$ in \cite{GJ}. They are instances of more general factors involved in the local Langlands conjectures: general $\gamma$-factors appear in connection with the lifting map in Langlands functoriality; see for instance the introduction of \cite{BK}.

The other facet of Representation Theory to which we shall connect these objects is the C*-algebraic picture of the tempered dual of reductive groups. In the past few years, various aspects of the tempered theory have been recast in the language of Noncommutative Geometry, thus making the machinery of operator algebras, known to be efficient in the general study of locally compact groups, available to study unitary representations of Lie groups. In particular, the study of parabolic induction by means of Hilbert modules and operator spaces allows to handle generalized principal series representations in families (see Section \ref{sec-ind-parab} below), which has proven useful in the study of adjunction properties for real groups (see \cite{CCH1,CCH2}).

In this picture, the analogues of the meromorphic factors used in Knapp-Stein theory to normalize intertwining integrals are distributions, conjecturally defined on the Levi components of the inducing parabolic subgroups. We observe in Section \ref{sec-resultat} that, when these distributions are known, they can be used to generate $\gamma$-factors in the sense of Braverman and Kazhdan, who conjectured in \cite{BK}, for any quasi-split reductive algebraic group $G$, the existence of a distribution acting in every irreducible representation of $G$ as the multiplication by the associated $\gamma$-factor.

In the case of $\GL(n)$, where the $\gamma$-factors coincide with those of Godement and Jacquet, this distribution is well-known, and our contribution is the observation that this distribution is essentially the inverse for the convolution product of the C*-algebraic normalizing distribution associated with a maximal parabolic subgroup of $\SL(n+1)$. The explicit formula relating the two distributions is given in Proposition \ref{resultat} and we indicate further connections with the Langlands-Shahidi method in Section \ref{sec-concl}.

\subsection*{Notation} In what follows $F$ is a local field, that is, $\R$, $\C$ or a finite extension of $\Q_p$ or of $\mathbb{F}_p((t))$. We fix a non-trivial continuous character $\psi:F\rightarrow\mathrm{U}(1)$. Then the map $a\mapsto\psi(a\,\cdot)$ is a group isomorphism between $F$ and its group of continuous unitary characters. We also fix an additive Haar measure $dx$ on $F$ so that the Fourier transform $\mathcal{F}_\psi$ defined by \[\mathcal{F}_\psi f(x)=\int_F\psi(xy)f(y)\,dy\] is an isometry of $L^2(F)$ with inverse \[f\longmapsto[x\mapsto\int_F\psi(-xy)f(y)\,dy].\]
Throughout the paper, additive Haar measures are denoted by $dx$ and multiplicative Haar measures are denoted by $d^\times x$. In particular, we write $dg$ for the additive Haar measure on the space $\Mn(F)$ of $n\times n$ matrices with entries in $F$ and $d^\times g$ for the Haar measure on $\GL(n,F)$. The two measures are related by \[dg=|\det(g)|^n\,d^\times g\] where $|\,\cdot\,|$ denotes the norm on $F$ that coincides with the modular function associated with the chosen Haar measure. In the non-Archimedean case, it is specified by the relation $|\varpi|=q^{-1}$, where $\varpi$ is a uniformizer and $q$ is the cardinality of the residue field.

\section{Generating Godement-Jacquet factors and the inverse normalizing distribution}

\subsection{$\gamma$-factors and the generating problem} In 1972, Godement and Jacquet introduced zeta functions associated with representations of $\GL(n,F)$ see \cite{GJ,Jacquet_Corvallis,Jacquet96}, generalizing the work of J. Tate on the case of $n=1$ \cite{Tate}. These functions can be defined as follows: \begin{equation}\label{def-zeta}Z(\Phi,s,f_\pi)=\int_{\GL(n,F)}\Phi(g)f_\pi(g)|\det(g)|^s\,d^\times g,\end{equation}
where $s$ is a complex parameter\footnote{Here, we adopt the convention of \cite{BK}. The notation originally used in \cite{GJ} can be recovered by a simple shift, namely replacing $s$ with $s-\frac{n-1}{2}$.}, $\Phi$ is a Schwartz function on $\Mn(F)$ and $f_\pi$ is a matrix coefficient of an admissible representation $\pi$ of $G$, that is, a function of the form \[g\longmapsto \scal{\pi(g)\xi}{\eta}\] with $\xi$ in  the carrying space $V_\pi$ of $\pi$ and $\eta$ in its dual $V_\pi^*$. The convergence of the integral in \eqref{def-zeta} is guaranteed for $\Re(s)$ large enough by \cite[Theorem 3.3]{GJ} and the expression $Z(\Phi,s,f_\pi)$ depends meromorphically on $s$.

An important property of these functions is the functional equation they satisfy, namely: \begin{equation}\label{funct-eq-GJ}Z(\hat{\Phi},n-s,\check{f_\pi})=\gamma(s,\pi)\cdot Z(\Phi,s,f_\pi).
\end{equation}

Here, $\check{f_\pi}(g)=f_\pi(g^{-1})$, so that $\check{f_\pi}$ is a matrix coefficient of the contragredient representation $\pi^*$, and $\hat{\Phi}$ denotes the Fourier transform of $\Phi$. We recall that the Fourier transform on $\Mn(F)$ is the operator \[\begin{array}{ccc}\mathcal{S}(\Mn(F))&\longrightarrow&\mathcal{S}(\Mn(F))\\\Phi&\longmapsto&\hat{\Phi}\end{array}\]
with
\begin{equation}\label{Fourier}\hat{\Phi}(x)=\int_{\Mn(F)}\Phi(y)\psi(\tr(xy))\,dy.\end{equation}

The meromorphic functions $\gamma(s,\pi)$ that appear in the functional equation \eqref{funct-eq-GJ} are examples of $\gamma$-factors. These objects can be defined in a more general setting, which we briefly describe now.

\subsubsection{$\gamma$-factors}

Let $G$ be a split reductive group over $F$ and let $\rho$ be a representation of dimension $n$ of its Langlands dual $G^\vee$. To every irreducible representation $\pi$ of $G$, local factors $\gamma_{\psi,\rho}(s,\pi)$, $L_\rho(s,\pi)$ and $\varepsilon_{\psi,\rho}(s,\pi)$ are conjecturally attached, satisfying the relation \[\gamma_{\psi,\rho}(s,\pi)=\varepsilon_{\psi,\rho}(s,\pi)\frac{L_\rho(1-s,\pi^*)}{L_\rho(s,\pi)}.\]
It is worth noting that, unlike the factors $\gamma$ and $\varepsilon$, the functions $L_\rho$ do not depend on the choice of $\psi$ and that the $\gamma$-factors are meromorphic functions of $s$. They are in fact quotients of products of $\Gamma$ functions if $F$ is Archimedean, and rational functions of $q^s$ in the non-Archimedean case, where $q$ is the cardinality of the residue field.

The Godement-Jacquet factors $\gamma(s,\pi)$ involved in the functional equation \eqref{funct-eq-GJ} are obtained by considering $G=\GL(n,F)$ and letting $\rho=\mathrm{St}_n$ be the standard representation of $G^\vee=\GL(n,F)$ on $F^n$.

\begin{remark}
An explicit description of the Godement-Jacquet factors $\gamma$, $L$ and $\varepsilon$ in terms of the classification of irreducible representations of $\GL(n,F)$ can be found in \cite{KnappGJ} for the Archimedean case and in \cite{Kudla} in the non-Archimedean case.
\end{remark}

\begin{remark}\label{rk-Rankin-Selberg}
The Rankin-Selberg local factors (see \cite{JPSS83}) correspond to the case of $G=\GL(m)\times\GL(n)$ with $\rho=\mathrm{St}_m\otimes\mathrm{St}_n$.
\end{remark}

In \cite{BK}, A. Braverman and D. Kazhdan studied the problem of constructing the factors $\gamma_{\psi,\rho}$ explicitly. They suppose given a character $\sigma$ of $G$ such that the cocharacter \[\sigma^*:\mathbb{G}_m\rightarrow \mathrm{Z}(G^\vee)\] of the center of $G^\vee$ satisfies the relation $\rho\circ\sigma^*=\mathrm{Id}$. Under this assumption, one has \[\gamma_{\psi,\rho}(s,\pi)=\gamma_{\psi,\rho}(0,\pi|\sigma|^s)\] so that the factor $\gamma_{\psi,\rho}$ is determined by $\gamma_{\psi,\rho}(0,\pi)$, thought of as a meromorphic function on the space of irreducible representations of $G$ (see Section 3.4 in \cite{BK} for a precise definition of this notion).

\subsubsection{Generating distribution}\label{sec-gener}

According to Braverman and Kazhdan, $\gamma$-factors should be generated by a single distribution on $G$. Indeed, they prove in some cases the existence of a central stable distribution $\Delta_{\psi,\rho}$ on $G$ whose action in every irreducible representation $\pi$ of $G$ is given by multiplication by $\gamma_{\psi,\rho}$. More precisely, convolution with the distribution\footnote{This distribution is denoted by $\Phi_{\psi,\rho}$ in \cite{BK}.} $\Delta_{\psi,\rho}$ acts in $\pi$ by $\gamma_{\psi,\rho}(0,\pi)$ and in $\pi|\sigma|^s$ by $\gamma_{\psi,\rho}(s,\pi)$, that is, \begin{equation}
\label{eq-Delta}\tag{$\dagger$}\Delta_{\psi,\rho}*(f_\pi|\sigma|^s)=\gamma_{\psi,\rho}(s,\pi)\cdot f_\pi|\sigma|^s
\end{equation}
for every matrix coefficient $f_\pi$ of $\pi$.

\begin{remark}The shifted convention for the complex parameter used in \cite{BK} and here has the merit of leading to the relation \[\Delta_{\psi,\rho}=\Delta_{\psi,\rho_1}*\Delta_{\psi,\rho_2}\] for $\rho=\rho_1\oplus\rho_2$.
\end{remark}

In the case of $G=\GL(n)$ and $\rho=\mathrm{St}_n$, we let \[\sigma=\det:\GL(n)\longrightarrow\mathbb{G}_m\] so that the defining property \eqref{eq-Delta} becomes \begin{equation}
\label{eq-DeltaGJ}\tag{$\ddagger$}\Delta_{\psi,\mathrm{St}_n}*(f_\pi|\det|^s)=\gamma(s,\pi)\cdot f_\pi|\det|^s
\end{equation} with the notation of \eqref{funct-eq-GJ}.

The determination of a distribution $\Delta_{\psi,\mathrm{St}_n}$ satisfying \eqref{eq-DeltaGJ} is very straightforward (see Section \ref{sec-resultat} below). The purpose of the present note is to point out a natural relation with the approach to the representation theory of algebraic groups \textit{via} C*-algebras and Hilbert modules.

\subsection{C*-normalization for degenerate principal series of $\SL(n)$} 

The basic principle underlying the Noncommutative Geometry approach to the study of tempered representations of real reductive groups is that the tempered dual of a topological group can be understood as a \emph{noncommutative space} in the sense of A. Connes \cite{NCG}. If $G$ is a locally compact group, its reduced C*-algebra $\Csr(G)$ is the norm-closure of the convolution algebra $C_c(G)$ of continuous compactly supported functions on $G$ in the algebra of bounded operator on $L^2(G)$. One can show that the spectrum of $\Csr(G)$ coincides with the tempered dual of $G$, making this algebra a useful proxy to study tempered representations of $G$ (see \cite{Dix}).

Following Harish-Chandra's philosophy of cusp forms, particular attention has been devoted in recent work to the description of principal series representations in the C*-algebraic language. We shall recall here only what is needed to establish a connection with the problem of generating $\gamma$-factors. The reader is referred to \cite{Artmodules} and \cite{CCH1} for proofs and details.

\subsubsection{Parabolic induction in the C*-algebraic framework}\label{sec-ind-parab}

Let $G$ be a reductive group and $P$ a parabolic subgroup with Levi component $L$ and unipotent radical $N$ so that $P=LN$. The generalized principal series representations induced from $P$ can be described in terms of certain $\Csr(G)$-$\Csr(L)$-bimodules (or correspondences) $\Csr(G/N)$ obtained as completions\footnote{These modules were denoted by $\mathcal{E}(G/N)$ in \cite{Artmodules} where they were introduced, and by $\Csr(G/N)$ in \cite{CCH1} where they were analyzed in further detail.} of $C_c(G/N)$.

The most important property of the module $\Csr(G/N)$ is that its induces the principal series attached to $P$ in the sense of M. A. Rieffel \cite{IRCA}:

\begin{prop}The Hilbert $\Csr(L)$-module $\Csr(G/N)$ is equipped with an action of $\Csr(G)$ through compact operators and yields a unitary equivalence of $G$-representations \[\Csr(G/N)\otimes_{\Csr(L)}\mathcal{H}_\lambda\simeq\operatorname{Ind}_P^G\lambda\otimes1_N\] for any tempered unitary representation $\lambda$ of $L$ and $1_N$ the trivial representation of $N$.
\end{prop}

See \cite{Artmodules} and \cite{CCH1} for a proof and further discussion of this result. In view of this property, the bimodules $\Csr(G/N)$ are sometimes referred to as \emph{C*-algebraic universal principal series}, adopting the terminology used in the context of $p$-adic groups, for instance in the the geometric approach to second adjointness (see \cite{BezKaz}), which was a part of the original motivation to study these modules. See \cite{CCH2} for general adjunction properties of the C*-algebraic functors.

\subsubsection{Intertwining integrals and C*-normalization}

Another aspect of the C*-algebraic picture concerns the description of intertwining relations between principal series at the level of the Hilbert bimodules $\Csr(G/N)$. The approach to this problem adopted in \cite{C*entrelacSL2} and \cite{C*entrelacSLn} is directly inspired by the work of A. Knapp and E. M. Stein on intertwining integrals and their normalization \cite{KS1, KS2}.

With $G$ and $P$ as before, we consider the parabolic subgroup $\bar{P}$ opposite to $P$. It decomposes as $\bar{P}=L\bar{N}$, giving rise to another $\Csr(G)$-$\Csr(L)$-correspondence $\Csr(G/\bar{N})$. The C*-algebraic analogues of intertwining operators between principal series representations are defined as operators between correspondences:

\begin{definition}\label{def-C*entrelac}
A \emph{C*-algebraic intertwiner} is a bounded adjointable $\Csr(L)$-linear operator \[\Csr(G/N)\longrightarrow\Csr(G/\bar{N})\] that commutes with the $\Csr(G)$-actions.
\end{definition}

The basic observation made in \cite{Artmodules} is that the standard intertwining integral \begin{equation}\label{def-int-stand}\mathcal{I}(f)(g\bar{N})=\int_{\bar{N}}f(g\bar{n})\,d\bar{n}
\end{equation} converges for functions in $C_c(G/N)\subset\Csr(G/N)$. However, although this expression formally satisfies good equivariance properties with respect to convolution actions of $C_c(G)\subset\Csr(G)$ and $C_c(L)\subset\Csr(L)$, it does not extend to a C*-algebraic intertwiner in the sense of Definition \ref{def-C*entrelac}.

To circumvent this issue, we proposed in \cite{C*entrelacSL2,C*entrelacSLn} a framework to adapt the normalization procedure used in Knapp-Stein theory.

\begin{definition}[C*-normalization]
A map $\U:\Csr(G/N)\longrightarrow\Csr(G/\bar{N})$ is said to \emph{normalize} the standard intertwining integral $\mathcal{I}$ in \eqref{def-int-stand} if
\begin{itemize}
\item[(i)]$\U$ is a C*-algebraic intertwiner that preserves $\Csr(L)$-norms;
\item[(ii)]there exists a distribution $\Gamma$ on $L$ such that \[\U=\mathcal{I}\circ C_\Gamma\] on a dense subspace of functions in $\Csr(G/N)$.
\end{itemize}
Here, $C_\Gamma$ denotes the convolution operator $f\mapsto f*\Gamma$.
\end{definition}

A distribution $\Gamma$ satisfying (ii) in the definition above is said to \emph{normalize} the standard integral $\mathcal{I}$.

The problem of normalizing standard intertwining integrals in the above sense has been solved so far in the particular case of maximal parabolic subgroup of the the special linear group. Both the unitary $\U$ and the normalizing distribution have been calculated explicitly, but we shall only need the latter for our current purposes:

\begin{prop}\cite{C*entrelacSLn}\label{prop-Gamma}
Let $G=\SL(n+1,F)$ and $P$ be a maximal parabolic subgroup $P$ with Levi component $L\simeq\GL(n,F)$. The distribution \[\Gamma=|\det(g)|^{\frac{n+1}{2}}\psi(\tr(g^{-1}))\,d^\times g\]
normalizes $\mathcal{I}$.
\end{prop}

\begin{remark}
The existence of C*-algebraic intertwiners in the sense of Definition \ref{def-C*entrelac} is crucial in the study of the structure of the reduced C*-algebra $\Csr(G)$. It can be obtained in general by reformulating the results of \cite{KS1,KS2}; see Theorem 6.1 in \cite{CCH1}.
\end{remark}

\subsection{Relationship in the case of $\GL(n)$}\label{sec-resultat}

We are now in a position to establish a connection between the problem of explicitly generating $\gamma$-factors in the sense of Braverman-Kazhdan and that of calculating normalizing distributions for C*-algebraic intertwining integrals. We consider the case of $\GL(n)$, where both problems have known answers.

We shall say that two distributions $S$ and $T$ on a group are \emph{inverse} to each other if their convolution is the Dirac measure at the neutral element, that is, \[S*T=\delta_{\mathrm{Id}}.\] Moreover, if $T$ is a distribution on $\GL(n,F)$ of the form $u(g)\,d^\times g$, we write \[\tilde{T}=u(-g)\,d^\times g.\] 

\begin{prop}\label{resultat}The distribution $\Delta_{\psi,\mathrm{St}_n}$ generating the Godement-Jacquet $\gamma$-factors is related to the C*-algebraic normalizing distribution $\Gamma$ associated to a parabolic subgroup of $\SL(n+1,F)$ with Levi component $\GL(n,F)$ by \[\Delta_{\psi,\mathrm{St}_n}=|\det|^{\frac{n+1}{2}}\,\tilde{\Gamma}^{-1}.\]
\end{prop}

\begin{proof} The expression of $\Gamma$ given in Proposition \ref{prop-Gamma} shows that \[\tilde{\Gamma}=|\det(g)|^{\frac{n+1}{2}}\psi(\tr(-g^{-1}))\,d^\times g\]
A straightforward calculation shows that the inverse for convolution of a distribution of the form \[|\det(g)|^\alpha\,\psi(\tr(-g^{-1}))\,d^{\times }g\] is \[|\det(g)|^{n-\alpha}\,\psi(\tr(g))\,d^{\times }g.\] It follows that \begin{equation}\label{eq-Delta-Gamma}|\det|^{\frac{n+1}{2}}\,\tilde{\Gamma}^{-1}=|\det(g)|^n\,\psi(\tr(g))\,d^{\times }g\end{equation} and it suffices to check that this distribution satisfies the Braverman-Kazhdan relation \eqref{eq-DeltaGJ}.

Indeed, given an irreducible representation $\pi$ of $\GL(n,F)$ and a matrix coefficient $f_\pi$, the functional equation \eqref{funct-eq-GJ} can be written as \begin{multline}\label{eq-eqfonct-dev}\int_{\GL(n,F)}\hat{\Phi}(g)f_\pi(g^{-1})|\det(g)|^{n-s}\,d^\times g\\=\gamma(s,\pi)\int_{\GL(n,F)}\Phi(g)f_\pi(g)|\det(g)|^s\,d^\times g.\end{multline} By definition of the Fourier transform \eqref{Fourier}, the left-hand side equals \[\int_{\GL(n,F)}\int_{\Mn}\Phi(x)\,\psi(\tr(xy))\,f_\pi(y^{-1})\,|\det(y)|^{n-s}\,dx\,d^\times y\]
and since $dx = |\det(x)|^n\,d^\times x$, the fact that \eqref{eq-eqfonct-dev} holds for any $\Phi$ is equivalent to having \begin{multline}\int_{\GL(n,F)}|\det(x)|^n\,\psi(\tr(xy))\,f_\pi(y^{-1})\,|\det(y)|^{n-s}\,d^\times y\\=\gamma(s,\pi)\cdot f_\pi(x)|\det(x)|^s\end{multline} for any matrix coefficient $f_\pi$. Letting $g=xy$, the left-hand side becomes \[\int_{\GL(n,F)}|\det(g)|^n\,\psi(\tr(g))\,\left(f_\pi\cdot|\det|^s\right)(g^{-1}x)\,d^\times z,\]
that is, according to \eqref{eq-Delta-Gamma}\[\left(|\det|^{\frac{n+1}{2}}\,\tilde{\Gamma}^{-1}*(f_\pi\cdot|\det|^s)\right)(x).\]
This shows that the functional equation \eqref{eq-eqfonct-dev} is equivalent to \[|\det|^{\frac{n+1}{2}}\,\tilde{\Gamma}^{-1}*(f_\pi\cdot|\det|^s)=\gamma(s,\pi)\cdot f_\pi|\det|^s\] which is exactly the relation \eqref{eq-DeltaGJ}, thus concluding the proof.
\end{proof}

\begin{remark}
A similar method to normalize intertwiners (at the level of $L^2$ spaces) by convolution was used in \cite{BK02}. See in particular Example 4.4.
\end{remark}

\section{Concluding remarks}\label{sec-concl}

To conclude this note, we shall point out that the existence of a relationship between $\gamma$-functions and normalization factors for intertwining integrals has been well-known for a long time. In particular the observation of Proposition \ref{resultat} can be interpreted as a consequence of the Langlands-Shahidi method (see \cite{Shahidi_Eisenstein} for a general exposition).

Assume that $F$ is non-Archimedean and consider the parabolic subgroup $P$ of type $(p,q)$ in $G=\SL(p+q,F)$, that is, the group of block upper-triangular matrices of the form \[\left[\begin{array}{cc}\alpha&*\\0&\beta\end{array}\right]\] with $\alpha$ in $\GL(p,F)$ and $\beta$ in $\GL(q,F)$ satisfying $\det(\alpha)\cdot\det(\beta)=1$. Denote by $L$ the Levi subgroup of block diagonal matrices and by $N$ the unipotent radical of $P$. The opposite parabolic subgroup containing $L$ is denoted by $\bar{P}=L\bar{N}$ and $\bar{N}$ consists of matrices of the form \[\left[\begin{array}{cc}\mathrm{Id}_p&0\\ *&\mathrm{Id}_q\end{array}\right].\]

In this setting, the standard integral \eqref{def-int-stand} defines an unbounded operator \[A:L^2(G/N)\longrightarrow L^2(G/\bar{N})\] and F. Shahidi has proved in \cite{Shahidi84} that the Rankin-Selberg $\gamma$-factors (called \emph{local coefficients for $P$} there) defined in \cite{JPSS83} and mentioned in Remark \ref{rk-Rankin-Selberg} above normalize this operator in the sense that the resulting operator $\mathcal{A}$ is equivariant with respect to the left and right actions of $G$ and $L$ on the spaces $L^2(G/N)$ and $L^2(G/\bar{N})$ and is compatible with the Whittaker functionals defined on the Schwartz Spaces of $G/N$ and $G/\bar{N}$.
The operator $\mathcal{A}$ is obtained as the convolution of $A$ and the distribution on $L$ that generates the Rankin-Selberg factors as in \eqref{eq-Delta}.

\begin{remark}
A comparison between the Braverman-Kazhdan approach and the Langlands-Shahidi method can be found in the recent paper \cite{ShahidiWWL}. See also \cite{ShahidiVinberg} and similar questions discussed in connection with the doubling method in \cite{GetzLiu}.
\end{remark}

\begin{remark}
Concerning real groups, results were obtained in \cite{Shahidi85} and \cite{Arthur89}, where unitarity of $\mathcal{A}$ is established.
\end{remark}

The case treated in the C*-algebraic framework in Proposition \ref{prop-Gamma} above corresponds to the situation when $p=1$. The result in Proposition \ref{resultat} suggests that studying intertwining relations at the level of C*-algebraic universal principal series and performing C*-algebraic normalization explicitly may help shed some light on the problem of generating $\gamma$-factors in the sens of Braverman and Kazhdan. Of particular interest is the question of determining an explicit formula for the kernel of the normalized operator $\mathcal{A}$ or, equivalently, for the central distribution satisfying \eqref{eq-Delta} beyond the Godement-Jacquet case.

\subsection*{Acknowledgements} The remark presented here was first made in the context of a seminar held in the Mathematics Department of the University of Orl\'eans, involving A. Alvarez, P. Julg and V. Lafforgue. The author thanks Vincent Lafforgue for many enlightening discussions about $\gamma$-factors and Wen-Wei Li for bringing recent developments to his attention.


\bibliographystyle{amsalpha}
\bibliography{biblio}

\end{document}